\documentclass{amsart}
\usepackage{amsmath,amssymb}

\newcommand{\dotminus}{\mathop{\mbox{$-^{\hspace{-.5em}\cdot}\,\,$}}}
\newcommand{\lep}{<\negmedspace\varepsilon_0}
\newcommand{\N}{\mathsf{N}}
\newcommand{\s}{\sigma}
\newcommand{\T}{\mathcal{T}}
\newcommand{\ta}{\tau}

\newtheorem{lemma}{Lemma}[section]
\newtheorem{theorem}[lemma]{Theorem}
\newtheorem*{statman}{Statman's Type-Reducibility Theorem}

\numberwithin{equation}{section}

\begin{document}

\title[On the definability of functionals in G\"odel's theory $\T$]{On the definability of functionals\\
in G\"odel's theory $\T$}

\author{Matthew~P. Szudzik}

\date{10 October 2014}

\begin{abstract}
G\"odel's theory $\T$ can be understood as a theory of the simply-typed lambda calculus that is extended to include the constant $0_\N$, the successor function $\mathrm{S}_+$, and the operator $\mathrm{R}_\ta$ for primitive recursion on objects of type $\ta$.  It is known that the functions from non-negative integers to non-negative integers that can be defined in this theory are exactly the $\lep$-recursive functions of non-negative integers.  As an extension of this result, we show that when the domain and codomain are restricted to pure closed normal forms, the \emph{functionals} of arbitrary type that are definable in $\T$ can be encoded as $\lep$-recursive functions.
\end{abstract}
\maketitle

\section{Introduction}

For the formalization of his \textit{Dialectica} interpretation of intuitionistic arithmetic, G\"odel~\cite{kG58} introduced the theory $\T$.  It was later shown~\cite{aD68,sH68,wH70} that $\T$ can be formalized as an extension of the simply-typed lambda calculus.\footnote{Readers unfamiliar with the simply-typed lambda calculus should consult reference~\cite{hB81}.}
In this formalization, the terms of the theory $\T$ are simply-typed lambda terms with ground type $\N$, extended to include the constants $0_\N^\N$, $\mathrm{S}_+^{\N\to\N}$, and $\mathrm{R}_\ta^{\ta\to(\ta\to\N\to\ta)\to\N\to\ta}$ for each type $\ta$.  We use superscripts to denote the types of terms, freely omitting the superscript when the type can be deduced from the term's context or when the type is unimportant.  The formulas of $\T$ are equations between terms, with formulas of the following forms taken as axioms for each type $\ta$
\begin{align}
\mathrm{R}_\ta\,A\,B\,0_\N&=A \notag\\
\mathrm{R}_\ta\,A\,B\,(\mathrm{S}_+\,C)&=B\,(\mathrm{R}_\ta\,A\,B\,C)\,C \notag
\end{align}
where $A$, $B$, and $C$ are metavariables for terms of types $\ta$, $\ta\to\N\to\ta$, and $\N$, respectively.  The rules of inference of $\T$ are the rules of $\beta\eta$-conversion and the rules of substitution of equality.\footnote{An alternative formalization of $\T$ extends the simply-typed lambda calculus to include the constants $0_\N^\N$, $\mathrm{S}_+^{\N\to\N}$, and $\mathrm{Iter}_\ta^{\ta\to(\ta\to\ta)\to\N\to\ta}$ for each type $\ta$, and takes the formulas of the following forms as axioms.
\begin{align}
\mathrm{Iter}_\ta\,A^\ta\,B^{\ta\to\ta}\,0_\N&=A \notag\\
\mathrm{Iter}_\ta\,A^\ta\,B^{\ta\to\ta}\,(\mathrm{S}_+\,C^\N)&=B\,(\mathrm{Iter}_\ta\,A\,B\,C) \notag
\end{align}
The rules of inference for this alternative formalization are the rules of $\beta\eta$-conversion and the rules of substitution of equality.  Both formalizations are equivalent because $\mathrm{Iter}_\ta$ can be expressed in terms of $\mathrm{R}_\ta$ as $\lambda a^\ta b^{\ta\to\ta}.\,\mathrm{R}_\ta\,a\,(\lambda x^\ta y^\N.\,b\,x)$, and because $\mathrm{R}_\ta$ can be expressed in terms of $\mathrm{Iter}_{\ta\times\N}$ as
\begin{align}
\lambda a^\ta b^{\ta\to\N\to\ta}c^\N.\,\mathrm{D}_{1,\ta,\N}\,\bigl(\mathrm{Iter}_{\ta\times\N}\,(\mathrm{D}_{0,\ta,\N}\,a\,0_\N)\,(\mathrm{H}_\ta\,b)\,c\bigr) \notag
\end{align}
where
\begin{align}
\mathrm{H}_\ta^{(\ta\to\N\to\ta)\to\ta\times\N\to\ta\times\N}=\lambda x^{\ta\to\N\to\ta}y^{\ta\times\N}.\,\mathrm{D}_{0,\ta,\N}\,\bigl(x\,(\mathrm{D}_{1,\ta,\N}\,y)\,(\mathrm{D}_{2,\ta,\N}\,y)\bigr)\,\bigl(\mathrm{S}_+\,(\mathrm{D}_{2,\ta,\N}\,y)\bigr) \notag
\end{align}
That is, higher-type primitive recursion (as characterized by $\mathrm{R}_\ta$) is equivalent to higher-type iteration (as characterized by $\mathrm{Iter}_\ta$) in the context of the simply-typed lambda calculus with $\beta\eta$-conversion.  See Section~\ref{s:seq} for a discussion of the terms $\mathrm{D}_{0,\ta,\N}$, $\mathrm{D}_{1,\ta,\N}$, and $\mathrm{D}_{2,\ta,\N}$.}  For any terms $A$ and $B$ in the language of $\T$, we write $\T\vdash A=B$ to denote that the equation $A=B$ is provable in $\T$.  We say that a term is a $\beta\eta\T$-normal form if and only if that term is a $\beta\eta$-normal form which, for each type $\ta$, contains no subterms of the form $\mathrm{R}_\ta\,A\,B\,0_\N$ or $\mathrm{R}_\ta\,A\,B\,(\mathrm{S}_+\,C)$.  Since $\T$ has the Church-Rosser property and is strongly normalizing~\cite{hB81,BDS11}, $\T\vdash A=B$ if and only if $A$ and $B$ have the same $\beta\eta\T$-normal forms.\footnote{G\"odel did not clearly define equality between higher-type terms in $\T$.  He only required that equality ``be understood as intensional or definitional equality''~\cite{kG90}.  Most formalizations of $\T$ in the simply-typed lambda calculus take equality to mean $\beta\T$-equality, omitting $\eta$-conversion as a rule of inference.  But we require $\beta\eta\T$-equality for Curry's pairing function in Section~\ref{s:seq} and for Statman's Type-Reducibility Theorem in Section~\ref{s:statman}.  This formalization of $\T$ is not uncommon (see references~\cite{hB74,hB81}, for example).  A survey of several other commonly-used formalizations of equality in $\T$ is contained in reference~\cite{AF98}.}

The closed terms of type $\N$ in the language of $\T$ are called \emph{numerals}.  Each numeral has a $\beta\eta\T$-normal form
\begin{align}
\overbrace{\mathrm{S}_+\,(\mathrm{S}_+\,(\cdots\,(\mathrm{S}_+}^n 0_\N)\cdots)) \notag
\end{align}
where the successor function $\mathrm{S}_+$ is applied $n$ many times to $0_\N$, and we abbreviate any such term as $\overline{n}$.

A function $f$ from closed $\beta\eta\T$-normal forms of types $\s_1$, $\s_2$,~$\ldots$~,~$\s_n$ to closed $\beta\eta\T$-normal forms of type $\ta$ is said to be \emph{defined} by a closed term $F^{\s_1\to\s_2\to\cdots\to\s_n\to\ta}$ in the language of $\T$ if and only if
\begin{align}
\T\vdash F\,A_1^{\s_1}\,A_2^{\s_2}\,\ldots\,A_n^{\s_n}=B^\ta \notag
\end{align}
whenever
\begin{align}
f\,A_1\,A_2\,\ldots\,A_n=B \notag
\end{align}
is true.  For example, because the constant $\mathrm{R}_\N$ denotes the operation of primitive recursion, every primitive recursive function of non-negative integers can be defined in $\T$, using the numerals $\overline{n}$ to represent the non-negative integers $n$.  Indeed, it has been shown~\cite{gK59,gK59b,wT65}
that the closed terms of type $\N\to\N\to\cdots\to\N$ in the language of $\T$ define exactly the $\lep$-recursive functions of non-negative integers.  In Section~\ref{s:func} we will show that when the $\beta\eta\T$-normal forms of types $\s$ and $\ta$ are restricted to be \emph{pure} closed $\beta\eta$-normal forms (that is, closed normal forms that do not contain any of the constants), then each functional of type $\s\to\ta$ that can be defined in G\"odel's theory $\T$ can be encoded as a $\lep$-recursive function of non-negative integers.  This result can naturally be extended to functionals of more than one argument.

\section{Examples of Primitive Recursion in $\T$}

Every primitive recursive function can be defined in $\T$.  For example, addition, multiplication, and predecessor are defined as follows.
\begin{align}
\mathrm{Add}^{\N\to\N\to\N}&=\lambda x^\N.\,\mathrm{R}_\N\,x\,(\lambda a^\N b^\N.\,\mathrm{S}_+\,a) \notag\\
\mathrm{Mult}^{\N\to\N\to\N}&=\lambda x^\N.\,\mathrm{R}_\N\,0_\N\,(\lambda a^\N b^\N.\,\mathrm{Add}\,a\,x) \notag\\
\mathrm{Pred}^{\N\to\N}&=\mathrm{R}_\N\,0_\N\,(\lambda a^\N b^\N.\,b) \notag
\end{align}
We write $A+B$ and $A\times B$ as abbreviations for $\mathrm{Add}\,A\,B$ and $\mathrm{Mult}\,A\,B$, respectively.  We define
\begin{align}
\mathrm{Monus}^{\N\to\N\to\N}=\lambda x^\N.\,\mathrm{R}_\N\,x\,(\lambda a^\N b^\N.\,\mathrm{Pred}\,a) \notag
\end{align}
and we write $A\dotminus B$ as an abbreviation for $\mathrm{Monus}\,A\,B$.  Note that for all non-negative integers $m$ and $n$, if $m\geq n$ then
\begin{align}
\T\vdash\overline{m}\dotminus\overline{n}=\overline{m-n} \notag
\end{align}
Otherwise, if $m<n$ then
\begin{align}
\T\vdash\overline{m}\dotminus\overline{n}=\overline{0} \notag
\end{align}
We write $\lvert A-B\rvert$ as an abbreviation for $(A\dotminus B)+(B\dotminus A)$.

The conditional function is defined as
\begin{align}
\mathrm{Cond}^{\N\to\N\to\N\to\N}=\lambda x^\N y^\N.\,\mathrm{R}_\N\,x\,(\lambda a^\N b^\N.\,y) \notag
\end{align}
For each non-negative integer $n$, if $n=0$ then
\begin{align}
\T\vdash\mathrm{Cond}\,x^\N\,y^\N\,\overline{n}=x \notag
\end{align}
Alternatively, if $n\ne0$ then
\begin{align}
\T\vdash\mathrm{Cond}\,x^\N\,y^\N\,\overline{n}=y \notag
\end{align}

Functionals can also be defined by primitive recursion.  For example, the summation functional is defined by
\begin{align}
\mathrm{Sum}^{\N\to(\N\to\N)\to\N}=\lambda x^\N f^{\N\to\N}.\,\mathrm{R}_\N\,0_\N\,\bigl(\lambda a^\N b^\N.\,a+(f\,b)\bigr)\,(\mathrm{S}_+\,x) \notag
\end{align}
It is common practice to write $\sum_{i=0}^nF\,\overline{i}$ as an abbreviation for $\mathrm{Sum}\,\overline{n}\,F$, where $i$ is a dummy variable.  Similarly, a functional for bounded maximization is defined by
\begin{align}
\mathrm{Max}_\leq^{\N\to(\N\to\N)\to\N}=\lambda x^\N f^{\N\to\N}.\,\mathrm{R}_\N\,0_\N\,\bigl(\lambda a^\N b^\N.\,\mathrm{Cond}\,b\,a\,(f\,b)\bigr)\,(\mathrm{S}_+\,x) \notag
\end{align}
Note that for each closed term $F^{\N\to\N}$ in the language of $\T$ and for each non-negative integer $n$, if $m$ is the largest non-negative integer less than or equal to $n$ such that $\T\vdash F\,\overline{m}=\overline{0}$, then
\begin{align}
\T\vdash\mathrm{Max}_\leq\,\overline{n}\,F=\overline{m} \notag
\end{align}
Otherwise, if no such $m$ exists, then
\begin{align}
\T\vdash\mathrm{Max}_\leq\,\overline{n}\,F=\overline{0} \notag
\end{align}
Division can be defined in terms of bounded maximization.
\begin{align}
\mathrm{Div}^{\N\to\N\to\N}=\lambda x^\N y^\N.\,\mathrm{Max}_\leq\,x\,\bigl(\lambda a^\N.\,(a\times y)\dotminus x\bigr) \notag
\end{align}
We write $\lfloor A/B\rfloor$ as an abbreviation for $\mathrm{Div}\,A\,B$.

\section{Pairs and Finite Sequences of Terms} \label{s:seq}

A variant of Cantor's pairing function~\cite{gC78} can be defined as follows.
\begin{align}
\mathrm{P}_0^{\N\to\N\to\N}=\lambda x^\N y^\N.\,\bigl\lfloor\bigl(x\times\bigl(x+\overline{3}\bigr)+y\times\bigl(y+\overline{1}\bigr)+\overline{2}\times x\times y\bigr)/\overline{2}\bigr\rfloor \notag
\end{align}
We write $\langle A,B\rangle$ as an abbreviation for $\mathrm{P}_0\,A\,B$.  For each non-negative integer $n$ there is exactly one pair of non-negative integers $m_1$ and $m_2$ such that
\begin{align}
\T\vdash\langle\overline{m_1},\overline{m_2}\rangle=\overline{n} \notag
\end{align}
And since $2m_1\leq m_1(m_1+3)$ and $2m_2\leq m_2(m_2+1)$ for all non-negative integers $m_1$ and $m_2$, it follows from the definition of $\mathrm{P}_0$ that $m_1\leq n$ and $m_2\leq n$.  (In fact, if $m_1\ne0$ then $m_1<n$ and $m_2<n$.)  Therefore, if we define
\begin{align}
\mathrm{P}_1^{\N\to\N}&=\lambda z^\N.\,\mathrm{Sum}\,z\,\bigl(\lambda y^\N.\,\mathrm{Max}_\leq\,z\,\bigl(\lambda x^\N.\,\bigl\lvert z-\langle x,y\rangle\bigr\rvert\bigr)\bigr) \notag\\
\mathrm{P}_2^{\N\to\N}&=\lambda z^\N.\,\mathrm{Sum}\,z\,\bigl(\lambda x^\N.\,\mathrm{Max}_\leq\,z\,\bigl(\lambda y^\N.\,\bigl\lvert z-\langle x,y\rangle\bigr\rvert\bigr)\bigr) \notag
\end{align}
then
\begin{align}
\T\vdash\mathrm{P}_1\,\langle\overline{m_1},\overline{m_2}\rangle&=\overline{m_1} \notag\\
\T\vdash\mathrm{P}_2\,\langle\overline{m_1},\overline{m_2}\rangle&=\overline{m_2} \notag
\end{align}
for each pair of non-negative integers $m_1$ and $m_2$.

Now, note that for each type $\ta$ there is a non-negative integer $n$ and there are types $\ta_1$, $\ta_2$,~$\ldots$~,~$\ta_n$ such that $\ta=\ta_1\to\ta_2\to\cdots\to\ta_n\to\N$.  Given any two types
\begin{align}
\s=&\s_1\to\s_2\to\cdots\to\s_m\to\N \notag\\
\ta=&\ta_1\to\ta_2\to\cdots\to\ta_n\to\N \notag
\end{align}
we define
\begin{align}
\s\times\ta=\s_1\to\s_2\to\cdots\to\s_m\to\ta_1\to\ta_2\to\cdots\to\ta_n\to(\N\to\N\to\N)\to\N \notag
\end{align}
Moreover, for each pair of types $\s$ and $\ta$, there are closed terms $\mathrm{D}_{0,\s,\ta}^{\s\to\ta\to\s\times\ta}$, $\mathrm{D}_{1,\s,\ta}^{\s\times\ta\to\s}$, and $\mathrm{D}_{2,\s,\ta}^{\s\times\ta\to\ta}$ such that
\begin{align}
\T\vdash\mathrm{D}_{1,\s,\ta}\,(\mathrm{D}_{0,\s,\ta}\,x^\s\,y^\ta)&=x \notag\\
\T\vdash\mathrm{D}_{2,\s,\ta}\,(\mathrm{D}_{0,\s,\ta}\,x^\s\,y^\ta)&=y \notag
\end{align}
Reference~\cite{hB74} provides explicit definitions for these terms.  $\mathrm{D}_{0,\s,\ta}$ is commonly known as Curry's pairing function.  We write $\{A_1^{\ta_1},A_2^{\ta_2}\}$ as an abbreviation for $\mathrm{D}_{0,\ta_1,\ta_2}\,A_1^{\ta_1}\,A_2^{\ta_2}$, and we write $\{A_1^{\ta_1},A_2^{\ta_2},\ldots,A_n^{\ta_n}\}$ as an abbreviation for the term
\begin{align}
\{A_1^{\ta_1},\{A_2^{\ta_2},\{\ldots,\{A_{n-1}^{\ta_{n-1}},A_n^{\ta_n}\}\ldots\}\}\} \notag
\end{align}
of type $\ta_1\times\ta_2\times\cdots\times\ta_n$.  Note that for each type $\ta=\ta_1\times\ta_2\times\cdots\times\ta_n$ such that $n>1$, and for each positive integer $i\leq n$, there is a closed term $\mathrm{D}_{i,\ta_1,\ta_2,\ldots,\ta_n}^{\ta\to\ta_i}$ such that
\begin{align}
\T\vdash\mathrm{D}_{i,\ta_1,\ta_2,\ldots,\ta_n}\,\{x_1^{\ta_1},x_2^{\ta_2},\ldots,x_n^{\ta_n}\}=x_i \notag
\end{align}

Now, for each type $\s=\s_1\to\s_2\to\cdots\to\s_m\to\N$ such that $m>0$, define
\begin{align}
0_\s^\s=\lambda x_1^{\s_1}x_2^{\s_2}\ldots x_m^{\s_m}.\,0_\N \notag
\end{align}
and for each type $\ta$ define
\begin{align}
\mathrm{Cons}_\ta^{\ta\to(\N\to\ta)\to\N\to\ta}=\lambda x^\ta y^{\N\to\ta}.\,\mathrm{R}_\ta\,x\,(\lambda a^\ta.\,y) \notag
\end{align}
We write $\bigl[A_0^\ta,A_1^\ta,\ldots,A_n^\ta\bigr]$ as an abbreviation for the term
\begin{align}
\mathrm{Cons}_\ta\,A_0\,(\mathrm{Cons}_\ta\,A_1\,(\cdots\,(\mathrm{Cons}_\ta\,A_n\,0_{\N\to\ta})\cdots)) \notag
\end{align}
of type $\N\to\ta$.  Note that
\begin{align}
&\T\vdash\bigl[x_0^\ta,x_1^\ta,\ldots,x_n^\ta\bigr]\,\overline{i}=x_i \notag\\
&\T\vdash\bigl[x_n^\ta,\ldots,x_1^\ta,x_0^\ta\bigr]\,\overline{n-i}=x_i \notag
\end{align}
for each non-negative integer $i\leq n$.

\section{Enumerating Pure Closed $\beta\eta$-Normal Forms}

Let $\mathfrak{F}_A$ denote the set of free variables in the term $A$.
\begin{lemma} \label{l:nf}
If $A^\ta$ is a pure $\beta\eta$-normal form, then one of the following three conditions must hold.
\begin{enumerate}
\item \label{e:atom} $A^\ta$ is a variable.
\item \label{e:lambda} $A^\ta$ is of the form $\lambda V^{\s_1}.\,B^{\s_2}$, where $\ta=\s_1\to\s_2$ and $B$ is a pure $\beta\eta$-normal form with free variables in the set $\{V\}\cup\mathfrak{F}_A$.
\item \label{e:compound} $A^\ta$ is of the form
\begin{align}
V^{\s_n\to\s_{n-1}\to\cdots\to\s_1\to\ta}\,B_n^{\s_n}\,B_{n-1}^{\s_{n-1}}\,\cdots\,B_1^{\s_1} \notag
\end{align}
where $V$ is a member of $\mathfrak{F}_A$ and $B_n$, $B_{n-1}$,~$\ldots$~,~$B_1$ are pure $\beta\eta$-normal forms with free variables in the set $\mathfrak{F}_A$.
\end{enumerate}
\end{lemma}
\begin{proof}
Any pure term $A^\ta$ must either be a variable, be of the form $\lambda V^{\s_1}.\,B^{\s_2}$ with $\ta=\s_1\to\s_2$, or be of the form $C_1^{\s_1\to\ta}\,B_1^{\s_1}$.  Condition~\ref{e:atom} and condition~\ref{e:lambda} follow immediately from the first two cases.  In the third case, if $A^\ta=C_1^{\s_1\to\ta}\,B_1^{\s_1}$ and $A^\ta$ is a pure $\beta\eta$-normal form, then $C_1^{\s_1\to\ta}$ must either be a variable or of the form $C_2^{\s_2\to\s_1\to\ta}\,B_2^{\s_2}$.  Therefore, by induction, $A^\ta$ must be of the form
\begin{align}
V^{\s_n\to\s_{n-1}\to\cdots\to\s_1\to\ta}\,B_n^{\s_n}\,B_{n-1}^{\s_{n-1}}\,\cdots\,B_1^{\s_1} \notag
\end{align}
for some positive integer $n$, where $V$ is a variable.  Condition~\ref{e:compound} immediately follows.
\end{proof}

For each non-negative integer $d$, let $\mathfrak{S}_{A,d}$ denote the set of \emph{subterms} of \emph{depth} $d$ in $A$.  That is, define $\mathfrak{S}_{A,0}$ to be the singleton set that contains only the term $A$, and for each non-negative integer $d$ define $B\in\mathfrak{S}_{A,d+1}$ if and only if $\mathfrak{S}_{A,d}$ contains a term of the form $\lambda V.\,B$, $B\,C$, or $C\,B$.  The set
\begin{align}
\mathfrak{S}_A=\mathfrak{S}_{A,0}\cup\mathfrak{S}_{A,1}\cup\mathfrak{S}_{A,2}\cup\cdots \notag
\end{align}
is the set of all subterms of $A$.

Similarly, if $\ta=\N$ then define $\s$ to be a \emph{subtype} of $\ta$ if and only if $\s=\N$.  Otherwise, if $\ta=\ta_1\to\ta_2$ then define $\s$ to be a subtype of $\ta$ if and only if $\s$ is a subtype of $\ta_1$, $\s$ is a subtype of $\ta_2$, or $\s=\ta$.
\begin{lemma} \label{l:subterm}
If $B^\s$ is a subterm of a pure closed $\beta\eta$-normal form $A^\ta$, then $\s$ is a subtype of $\ta$ and the type of each free variable in $B$ is a subtype of $\ta$.
\end{lemma}
\begin{proof}
The proof is by induction on the depth of each subterm $B$ in $A$.  Suppose that $A^\ta$ is a pure closed $\beta\eta$-normal form.  As the base case, note that if $B^\s=A^\ta$ then $\s=\ta$ is a subtype of $\ta$ and $B$ has no free variables because $A$ is closed.  As the inductive hypothesis, suppose that $B^\s$ is a subterm of $A^\ta$ with $\s$ a subtype of $\ta$ and with the type of each member of $\mathfrak{F}_B$ a subtype of $\ta$.  Since $A$ is a pure $\beta\eta$-normal form, $B$ is a pure $\beta\eta$-normal form and $B$ satisfies one of the three conditions in Lemma~\ref{l:nf}.  In particular, if $B$ satisfies condition~\ref{e:atom}, then $B$ has no subterms except for itself.  Alternatively, if $B$ satisfies condition~\ref{e:lambda}, then $B^\s=\lambda V^{\s_1}.\,C^{\s_2}$ with $\s=\s_1\to\s_2$, and the free variables of $C$ are members of the set $\{V^{\s_1}\}\cup\mathfrak{F}_B$.  But because $\s=\s_1\to\s_2$ is a subtype of $\ta$ by the inductive hypothesis, $\s_1$ and $\s_2$ are subtypes of $\ta$.  So, the type of $C$ is a subtype of $\ta$ and the type of each free variable in $C$ is a subtype of $\ta$.  Finally, if $B$ satisfies condition~\ref{e:compound}, then
\begin{align} B^\s=V^{\s_n\to\s_{n-1}\to\cdots\to\s_1\to\s}\,C_n^{\s_n}\,C_{n-1}^{\s_{n-1}}\,\cdots\,C_1^{\s_1} \notag
\end{align}
where $V^{\s_n\to\s_{n-1}\to\cdots\to\s_1\to\s}$ is a variable.  But because $V$ is a free variable of $B$, it follows from the inductive hypothesis that $\s_n\to\s_{n-1}\to\cdots\to\s_1\to\s$ is a subtype of $\ta$, as are $\s_1$ and $\s_1\to\s$.  Hence, the type of $C_1$ is a subtype of $\ta$, as is the type of the subterm $V\,C_n\,C_{n-1}\,\cdots\,C_2$.  Of course, the free variables of these subterms are members of $\mathfrak{F}_B$, so the types of the free variables in these subterms are subtypes of $\ta$.
\end{proof}

Now, let $A^\ta$ be a pure closed $\beta\eta$-normal form and assume, without loss of generality, that any two distinct occurrences of $\lambda$ in $A$ bind variables with distinct names.\footnote{By convention, $\lambda x_1x_2\ldots x_m.\,B$ is an abbreviation for $\lambda x_1.\,(\lambda x_2.\,(\cdots(\lambda x_m.\,B)\cdots))$.  Hence, $\lambda x_1x_2\ldots x_m$ is an abbreviation for $m$ occurrences of $\lambda$.}
Let $\ta_1$, $\ta_2$,~$\ldots$~,~$\ta_n$ be all the subtypes of $\ta$ and let $\ta_1=\ta$.  For each term $B$ in $\mathfrak{S}_A$ define
\begin{align}
B\oslash A=
\begin{cases}
\bigl\langle\overline{0},\overline{d}\bigr\rangle &\text{if $(\lambda B.\,C)\in\mathfrak{S}_{A,d}$}\\
\bigl\langle\overline{j},\langle C\oslash A,D\oslash A\rangle\bigr\rangle &\text{if $B=C^{\ta_j\to\ta_i}\,D^{\ta_j}$}\\
\bigl\langle\overline{n+1},C\oslash A\bigr\rangle &\text{if $B=\lambda V.\,C$}
\end{cases} \notag
\end{align}
Note that for each pure closed $\beta\eta$-normal form $A^\ta$, $A\oslash A$ is a numeral that encodes $A^\ta$.  For example, if $A=\lambda x^{(\N\to\N)\to\N}.\,x\,(\lambda y^\N.\,y)$ and $\ta_3=\N\to\N$, then
\begin{align}
A\oslash A=\bigl\langle\overline{5},\bigl\langle\overline{3},\bigl\langle\bigl\langle\overline{0},\overline{0}\bigr\rangle,\bigl\langle\overline{5},\bigl\langle\overline{0},\overline{2}\bigr\rangle\bigr\rangle\bigr\rangle\bigr\rangle\bigr\rangle \notag
\end{align}

Next, define
\begin{multline}
\mathrm{A}_\ta^{\N\to\N\to\upsilon}=\lambda x^\N.\\
\,\mathrm{R}_{\N\to\upsilon}\,0_{\N\to\upsilon}\,\bigl(\lambda a^{\N\to\upsilon}b^\N.\,\mathrm{Cons}_\upsilon\,\{\mathrm{B}_{1,\ta}\,a\,b,\mathrm{B}_{2,\ta}\,a\,b,\ldots,\mathrm{B}_{n,\ta}\,a\,b\}\,a\bigr)\,(\mathrm{S}_+\,x) \notag
\end{multline}
and for each positive integer $i\leq n$ let
\begin{align}
\mathrm{B}_{i,\ta}^{(\N\to\upsilon)\to\N\to\upsilon_i}&=\lambda a^{\N\to\upsilon}b^\N.\,[\mathrm{J}_{0,i,\ta}\,a\,b,\mathrm{J}_{1,i,\ta}\,a\,b,\ldots,\mathrm{J}_{n+1,i,\ta}\,a\,b]\,(\mathrm{P}_1\,b) \notag\\
\mathrm{J}_{0,i,\ta}^{(\N\to\upsilon)\to\N\to\upsilon_i}&=\lambda a^{\N\to\upsilon}b^\N x_1^{\N\to\ta_1}x_2^{\N\to\ta_2}\ldots x_n^{\N\to\ta_n}y^\N.\,x_i\,(y\dotminus\overline{1}\dotminus\mathrm{P}_2\,b) \notag
\end{align}
where
\begin{align}
\upsilon_i=(\N\to\ta_1)\to(\N\to\ta_2)\to\cdots\to(\N\to\ta_n)\to\N\to\ta_i \notag
\end{align}
and $\upsilon=\upsilon_1\times\upsilon_2\times\cdots\times\upsilon_n$.  Furthermore, for each positive integer $j\leq n$, if there exists a positive integer $k\leq n$ such that $\ta_k=\ta_j\to\ta_i$ then define
\begin{multline}
\mathrm{J}_{j,i,\ta}^{(\N\to\upsilon)\to\N\to\upsilon_i}=\lambda a^{\N\to\upsilon}b^\N x_1^{\N\to\ta_1}\ldots x_n^{\N\to\ta_n}y^\N.\\
\,\Bigl(\mathrm{D}_{k,\upsilon_1,\ldots,\upsilon_n}\,\bigl(a\,\bigl(b\dotminus\overline{1}\dotminus\mathrm{P}_1\,(\mathrm{P}_2\,b)\bigr)\bigr)\,(\mathrm{Cons}_{\ta_1}\,0_{\ta_1}\,x_1)\,\cdots\,(\mathrm{Cons}_{\ta_n}\,0_{\ta_n}\,x_n)\,(\mathrm{S}_+\,y)\Bigr)\\
\,\Bigl(\mathrm{D}_{j,\upsilon_1,\ldots,\upsilon_n}\,\bigl(a\,\bigl(b\dotminus\overline{1}\dotminus\mathrm{P}_2\,(\mathrm{P}_2\,b)\bigr)\bigr)\,(\mathrm{Cons}_{\ta_1}\,0_{\ta_1}\,x_1)\,\cdots\,(\mathrm{Cons}_{\ta_n}\,0_{\ta_n}\,x_n)\,(\mathrm{S}_+\,y)\Bigr) \notag
\end{multline}
Otherwise, if no such $k$ exists, then define $\mathrm{J}_{j,i,\ta}=\lambda a^{\N\to\upsilon}b^\N.\,0_{\upsilon_i}$.   Similarly, if there exist positive integers $j\leq n$ and $k\leq n$ such that $\ta_i=\ta_j\to\ta_k$, then define
\begin{multline}
\mathrm{J}_{n+1,i,\ta}^{(\N\to\upsilon)\to\N\to\upsilon_i}=\lambda a^{\N\to\upsilon}b^\N x_1^{\N\to\ta_1}\ldots x_n^{\N\to\ta_n}y^\N z^{\ta_j}.\,\mathrm{D}_{k,\upsilon_1,\ldots,\upsilon_n}\\
\,\bigl(a\,(b\dotminus\overline{1}\dotminus\mathrm{P}_2\,b)\bigr)\,\bigl(\mathrm{Cons}_{\ta_1}\,(\mathrm{L}_{j,1}\,z)\,x_1\bigr)\,\cdots\,\bigl(\mathrm{Cons}_{\ta_n}\,(\mathrm{L}_{j,n}\,z)\,x_n\bigr)\,(\mathrm{S}_+\,y) \notag
\end{multline}
where
\begin{align}
\mathrm{L}_{j,l}^{\ta_j\to\ta_l}= \notag
\begin{cases}
\lambda z^{\ta_j}.\,z &\text{if $l=j$}\\
0_{\ta_j\to\ta_l} &\text{otherwise}
\end{cases}
\end{align}
for each positive integer $l\leq n$.  Otherwise, if no such $j$ and $k$ exist, then define $\mathrm{J}_{n+1,i,\ta}=\lambda a^{\N\to\upsilon}b^\N.\,0_{\upsilon_i}$.

\begin{lemma} \label{l:a}
For all non-negative integers $i$ and $j$,
\begin{align}
\T\vdash\mathrm{A}_\ta\,\overline{i+j}\;\overline{j}=\mathrm{A}_\ta\,\overline{i}\;\overline{0} \notag
\end{align}
\end{lemma}
\begin{proof}
The proof is by induction on $j$.  The base case, when $j=0$, is trivial.  As the inductive hypothesis, suppose that 
\begin{align}
\T\vdash\mathrm{A}_\ta\,\overline{i+j}\;\overline{j}=\mathrm{A}_\ta\,\overline{i}\;\overline{0} \notag
\end{align}
By the definition of $\mathrm{A}_\ta$ we have that
\begin{multline}
\T\vdash\mathrm{A}_\ta\,\overline{i+j+1}\;\overline{j+1}=\mathrm{Cons}_\upsilon\,\Bigl\{\mathrm{B}_{1,\ta}\,\bigl(\mathrm{A}_\ta\,\overline{i+j}\bigr)\,\overline{i+j+1},\ldots,\\
\mathrm{B}_{n,\ta}\,\bigl(\mathrm{A}_\ta\,\overline{i+j}\bigr)\,\overline{i+j+1}\Bigr\}\,\bigl(\mathrm{A}_\ta\,\overline{i+j}\bigr)\,\overline{j+1} \notag
\end{multline}
and by the definition of $\mathrm{Cons}_\upsilon$ we have that
\begin{align}
\T\vdash\mathrm{A}_\ta\,\overline{i+j+1}\;\overline{j+1}=\mathrm{A}_\ta\,\overline{i+j}\;\overline{j} \notag
\end{align}
Then, by the inductive hypothesis,
\begin{align}
\T\vdash\mathrm{A}_\ta\,\overline{i+j+1}\;\overline{j+1}=\mathrm{A}_\ta\,\overline{i}\;\overline{0} \notag
\end{align}
\end{proof}

\begin{theorem} \label{t:a}
Let $A^\ta$ be a pure closed $\beta\eta$-normal form and let $B^{\ta_i}$ be a member of $\mathfrak{S}_{A,d}$.  If for each positive integer $l\leq n$, $X_l^{\N\to\ta_l}$ is a term such that
\begin{align}
\T\vdash X_l\,\overline{d-1-e}=V^{\ta_l} \notag
\end{align}
whenever $V\in\mathfrak{F}_B$ and $(\lambda V.\,C)\in\mathfrak{S}_{A,e}$ for some term $C$, then
\begin{align}
\T\vdash\mathrm{D}_{i,\ta_1,\ta_2,\ldots,\ta_n}\,\bigl(\mathrm{A}_\ta\,(B\oslash A)\,\overline{0}\bigr)\,X_1\,X_2\,\ldots\,X_n\,\overline{d}=B \notag
\end{align}
\end{theorem}
\begin{proof}
The proof is by induction on $B\oslash A$.  For the base case, note that if $\T\vdash B\oslash A=\overline{0}$ then $B\oslash A=\bigl\langle\overline{0},\overline{0}\bigr\rangle$.  Therefore, $B$ is a variable and $(\lambda B.\,C)\in\mathfrak{S}_{A,0}$ for some term $C$.  Now,
\begin{multline}
\T\vdash\mathrm{A}_\ta\,(B\oslash A)\,\overline{0}=\\
\mathrm{R}_{\N\to\upsilon}\,0_{\N\to\upsilon}\,\bigl(\lambda a^{\N\to\upsilon}b^\N.\,\mathrm{Cons}_\upsilon\,\{\mathrm{B}_{1,\ta}\,a\,b,\ldots,\mathrm{B}_{n,\ta}\,a\,b\}\,a\bigr)\,\bigl(\mathrm{S}_+\,\overline{0}\bigr)\,\overline{0} \notag
\end{multline}
and so
\begin{align}
\T\vdash\mathrm{A}_\ta\,(B\oslash A)\,\overline{0}&=\mathrm{Cons}_\upsilon\,\{\mathrm{B}_{1,\ta}\,0_{\N\to\upsilon}\,\overline{0},\ldots,\mathrm{B}_{n,\ta}\,0_{\N\to\upsilon}\,\overline{0}\}\,0_{\N\to\upsilon}\,\overline{0} \notag\\
\T\vdash\mathrm{A}_\ta\,(B\oslash A)\,\overline{0}&=\{\mathrm{B}_{1,\ta}\,0_{\N\to\upsilon}\,\overline{0},\ldots,\mathrm{B}_{n,\ta}\,0_{\N\to\upsilon}\,\overline{0}\} \notag
\end{align}
Therefore,
\begin{align}
\T\vdash\mathrm{D}_{i,\ta_1,\ldots,\ta_n}\,\bigl(\mathrm{A}_\ta\,(B\oslash A)\,\overline{0}\bigr)\,X_1\,\ldots\,X_n\,\overline{d}&=\mathrm{B}_{i,\ta}\,0_{\N\to\upsilon}\,\overline{0}\,X_1\,\ldots\,X_n\,\overline{d} \notag\\
\T\vdash\mathrm{D}_{i,\ta_1,\ldots,\ta_n}\,\bigl(\mathrm{A}_\ta\,(B\oslash A)\,\overline{0}\bigr)\,X_1\,\ldots\,X_n\,\overline{d}&=\mathrm{J}_{0,i,\ta}\,0_{\N\to\upsilon}\,\overline{0}\,X_1\,\ldots\,X_n\,\overline{d} \notag\\
\T\vdash\mathrm{D}_{i,\ta_1,\ldots,\ta_n}\,\bigl(\mathrm{A}_\ta\,(B\oslash A)\,\overline{0}\bigr)\,X_1\,\ldots\,X_n\,\overline{d}&=X_i\,\bigl(\overline{d-1}\dotminus\mathrm{P}_2\,\overline{0}\bigr) \notag\\
\T\vdash\mathrm{D}_{i,\ta_1,\ldots,\ta_n}\,\bigl(\mathrm{A}_\ta\,(B\oslash A)\,\overline{0}\bigr)\,X_1\,\ldots\,X_n\,\overline{d}&=X_i\,\overline{d-1} \notag
\end{align}
But if $\T\vdash X_i\,\overline{d-1}=B$ then
\begin{align}
\T\vdash\mathrm{D}_{i,\ta_1,\ldots,\ta_n}\,\bigl(\mathrm{A}_\ta\,(B\oslash A)\,\overline{0}\bigr)\,X_1\,\ldots\,X_n\,\overline{d}=B \notag
\end{align}

As the inductive hypothesis, let $m$ be a non-negative integer and suppose, for all positive integers $i\leq n$ and all non-negative integers $d$, that if $B^{\ta_i}$ is a member of $\mathfrak{S}_{A,d}$ such that $B\oslash A$ is less than or equal to $\overline{m}$, then the statement of the theorem holds.  Now consider any term $B^{\ta_i}$ with $\T\vdash B\oslash A=\overline{m+1}$ and such that $B\in\mathfrak{S}_{A,d}$.  By the same sort of reasoning as in the base case, we have that
\begin{align}
\T\vdash\mathrm{D}_{i,\ta_1,\ldots,\ta_n}\,\bigl(\mathrm{A}_\ta\,(B\oslash A)\,\overline{0}\bigr)\,X_1\,\ldots\,X_n\,\overline{d}=\mathrm{B}_{i,\ta}\,(\mathrm{A}_\ta\,\overline{m}\,)\,(B\oslash A)\,X_1\,\ldots\,X_n\,\overline{d} \notag
\end{align}
Now, by Lemma~\ref{l:subterm} every subterm of $B$ must have a type that is a subtype of $\ta$ and have free variables with types that are subtypes of $\ta$.  Hence, there are three possibilities: $B^{\ta_i}$ is a variable, $B^{\ta_i}=C^{\ta_k}\,D^{\ta_j}$, or $B^{\ta_i}=\lambda V^{\ta_j}.\,C^{\ta_k}$.  We will consider each of these possibilities separately.  First, if $B$ is a variable then $B\oslash A=\bigl\langle\overline{0},\overline{e}\bigr\rangle$ and there must exist a term $C$ such that $(\lambda B.\,C)\in\mathfrak{S}_{A,e}$.  In this case,
\begin{align}
\T\vdash\mathrm{D}_{i,\ta_1,\ldots,\ta_n}\,\bigl(\mathrm{A}_\ta\,(B\oslash A)\,\overline{0}\bigr)\,X_1\,\ldots\,X_n\,\overline{d}&=\mathrm{J}_{0,i,\ta}\,(\mathrm{A}_\ta\,\overline{m}\,)\,(B\oslash A)\,X_1\,\ldots\,X_n\,\overline{d} \notag\\
\T\vdash\mathrm{D}_{i,\ta_1,\ldots,\ta_n}\,\bigl(\mathrm{A}_\ta\,(B\oslash A)\,\overline{0}\bigr)\,X_1\,\ldots\,X_n\,\overline{d}&=X_i\,\overline{d-1-e} \notag
\end{align}
And if $\T\vdash X_i\,\overline{d-1-e}=B$ then
\begin{align}
\T\vdash\mathrm{D}_{i,\ta_1,\ldots,\ta_n}\,\bigl(\mathrm{A}_\ta\,(B\oslash A)\,\overline{0}\bigr)\,X_1\,\ldots\,X_n\,\overline{d}=B \notag
\end{align}

Alternatively, if $B^{\ta_i}=C^{\ta_k}\,D^{\ta_j}$ for some positive integers $j\leq n$ and $k\leq n$, then $B\oslash A=\bigl\langle\overline{j},\langle C\oslash A,D\oslash A\rangle\bigr\rangle$ and
\begin{align}
\T\vdash\mathrm{D}_{i,\ta_1,\ldots,\ta_n}\,\bigl(\mathrm{A}_\ta\,(B\oslash A)\,\overline{0}\bigr)\,X_1\,\ldots\,X_n\,\overline{d}=\mathrm{J}_{j,i,\ta}\,(\mathrm{A}_\ta\,\overline{m}\,)\,(B\oslash A)\,X_1\,\ldots\,X_n\,\overline{d} \notag
\end{align}
Therefore,
\begin{multline}
\T\vdash\mathrm{D}_{i,\ta_1,\ldots,\ta_n}\,\bigl(\mathrm{A}_\ta\,(B\oslash A)\,\overline{0}\bigr)\,X_1\,\ldots\,X_n\,\overline{d}=\\
\Bigl(\mathrm{D}_{k,\upsilon_1,\ldots,\upsilon_n}\,\bigl(\mathrm{A}_\ta\,\overline{m}\,\bigl(\overline{m}\dotminus(C\oslash A)\bigr)\bigr)\,(\mathrm{Cons}_{\ta_1}\,0_{\ta_1}\,X_1)\,\cdots\,(\mathrm{Cons}_{\ta_n}\,0_{\ta_n}\,X_n)\,\overline{d+1}\Bigr)\\
\Bigl(\mathrm{D}_{j,\upsilon_1,\ldots,\upsilon_n}\,\bigl(\mathrm{A}_\ta\,\overline{m}\,\bigl(\overline{m}\dotminus(D\oslash A)\bigr)\bigr)\,(\mathrm{Cons}_{\ta_1}\,0_{\ta_1}\,X_1)\,\cdots\,(\mathrm{Cons}_{\ta_n}\,0_{\ta_n}\,X_n)\,\overline{d+1}\Bigr) \notag
\end{multline}
Note, by the definition of Cantor's pairing function, that $C\oslash A$ and $D\oslash A$ are less than or equal to $\overline{m}$ because
\begin{align}
\T\vdash\overline{m+1}=\bigl\langle\overline{j},\langle C\oslash A,D\oslash A\rangle\bigr\rangle \notag
\end{align}
and $j\ne 0$.  It then follows from Lemma~\ref{l:a} that
\begin{multline}
\T\vdash\mathrm{D}_{i,\ta_1,\ldots,\ta_n}\,\bigl(\mathrm{A}_\ta\,(B\oslash A)\,\overline{0}\bigr)\,X_1\,\ldots\,X_n\,\overline{d}=\\
\Bigl(\mathrm{D}_{k,\upsilon_1,\ldots,\upsilon_n}\,\bigl(\mathrm{A}_\ta\,(C\oslash A)\,\overline{0}\bigr)\,(\mathrm{Cons}_{\ta_1}\,0_{\ta_1}\,X_1)\,\cdots\,(\mathrm{Cons}_{\ta_n}\,0_{\ta_n}\,X_n)\,\overline{d+1}\Bigr)\\
\Bigl(\mathrm{D}_{j,\upsilon_1,\ldots,\upsilon_n}\,\bigl(\mathrm{A}_\ta\,(D\oslash A)\,\overline{0}\bigr)\,(\mathrm{Cons}_{\ta_1}\,0_{\ta_1}\,X_1)\,\cdots\,(\mathrm{Cons}_{\ta_n}\,0_{\ta_n}\,X_n)\,\overline{d+1}\Bigr) \notag
\end{multline}
But if for each positive integer $l\leq n$, $X_l^{\N\to\ta_l}$ is a term such that
\begin{align}
\T\vdash X_l\,\overline{d-1-e}=V^{\ta_l} \notag
\end{align}
whenever $V\in\mathfrak{F}_B$ and $(\lambda V.\,E)\in\mathfrak{S}_{A,e}$ for some term $E$, then
\begin{align}
\T\vdash \mathrm{Cons}_{\ta_l}\,0_{\ta_l}\,X_l\,\overline{d+1-1-e}=V^{\ta_l} \notag
\end{align}
whenever $V\in\mathfrak{F}_C$ and $(\lambda V.\,E)\in\mathfrak{S}_{A,e}$, and similarly for $D$.  Hence, by the inductive hypothesis,
\begin{align}
\T\vdash\mathrm{D}_{i,\ta_1,\ldots,\ta_n}\,\bigl(\mathrm{A}_\ta\,(B\oslash A)\,\overline{0}\bigr)\,X_1\,\ldots\,X_n\,\overline{d}&=C\,D \notag\\
\T\vdash\mathrm{D}_{i,\ta_1,\ldots,\ta_n}\,\bigl(\mathrm{A}_\ta\,(B\oslash A)\,\overline{0}\bigr)\,X_1\,\ldots\,X_n\,\overline{d}&=B \notag
\end{align}

The final case to be considered is when $B^{\ta_i}=\lambda V^{\ta_j}.\,C^{\ta_k}$ for some positive integers $j\leq n$ and $k\leq n$.  In this case, $B\oslash A=\bigl\langle\overline{n+1},C\oslash A\bigr\rangle$ and 
\begin{align}
\T\vdash\mathrm{D}_{i,\ta_1,\ldots,\ta_n}\,\bigl(\mathrm{A}_\ta\,(B\oslash A)\,\overline{0}\bigr)\,X_1\,\ldots\,X_n\,\overline{d}=\mathrm{J}_{n+1,i,\ta}\,(\mathrm{A}_\ta\,\overline{m}\,)\,(B\oslash A)\,X_1\,\ldots\,X_n\,\overline{d} \notag
\end{align}
Therefore,
\begin{multline}
\T\vdash\mathrm{D}_{i,\ta_1,\ldots,\ta_n}\,\bigl(\mathrm{A}_\ta\,(B\oslash A)\,\overline{0}\bigr)\,X_1\,\ldots\,X_n\,\overline{d}=\lambda z^{\ta_j}.\,\mathrm{D}_{k,\upsilon_1,\ldots,\upsilon_n}\\
\,\bigl(\mathrm{A}_\ta\,\overline{m}\,\bigl(\overline{m}\dotminus(C\oslash A)\bigr)\bigr)\,\bigl(\mathrm{Cons}_{\ta_1}\,(\mathrm{L}_{j,1}\,z)\,X_1\bigr)\,\cdots\,\bigl(\mathrm{Cons}_{\ta_n}\,(\mathrm{L}_{j,n}\,z)\,X_n\bigr)\,\overline{d+1} \notag
\end{multline}
As in the previous case, it follows from the definition of Cantor's pairing function that $C\oslash A$ is less than or equal to $\overline{m}$.  So, by Lemma~\ref{l:a} we have that
\begin{multline}
\T\vdash\mathrm{D}_{i,\ta_1,\ldots,\ta_n}\,\bigl(\mathrm{A}_\ta\,(B\oslash A)\,\overline{0}\bigr)\,X_1\,\ldots\,X_n\,\overline{d}=\lambda z^{\ta_j}.\,\mathrm{D}_{k,\upsilon_1,\ldots,\upsilon_n}\,\\
\bigl(\mathrm{A}_\ta\,(C\oslash A)\,\overline{0}\bigr)\,\bigl(\mathrm{Cons}_{\ta_1}\,(\mathrm{L}_{j,1}\,z)\,X_1\bigr)\,\cdots\,\bigl(\mathrm{Cons}_{\ta_n}\,(\mathrm{L}_{j,n}\,z)\,X_n\bigr)\,\overline{d+1} \notag
\end{multline}
And if for each positive integer $l\leq n$, $X_l^{\N\to\ta_l}$ is a term such that
\begin{align}
\T\vdash X_l\,\overline{d-1-e}=V^{\ta_l} \notag
\end{align}
whenever $V\in\mathfrak{F}_B$ and $(\lambda V.\,D)\in\mathfrak{S}_{A,e}$ for some term $D$, then
\begin{align}
\T\vdash \mathrm{Cons}_{\ta_l}\,0_{\ta_l}\,X_l\,\overline{d+1-1-e}=V^{\ta_l} \notag
\end{align}
whenever $V\in\mathfrak{F}_C$, $(\lambda V.\,D)\in\mathfrak{S}_{A,e}$, and $l\ne j$, because the free variables of type $\ta_l$ in $C$ are also the free variables of type $\ta_l$ in $B$ when $l\ne j$.  But $B$ has exactly one more bound variable of type $\ta_j$ than $C$.  Assume, without loss of generality, that this variable is $z^{\ta_j}$.  Then,
\begin{align}
\T\vdash \mathrm{Cons}_{\ta_j}\,z\,X_j\,\overline{d+1-1-e}=V^{\ta_j} \notag
\end{align}
whenever $V\in\mathfrak{F}_C$ and $(\lambda V.\,D)\in\mathfrak{S}_{A,e}$.  It immediately follows from the inductive hypothesis and from the definition of $\mathrm{L}_{j,l}$ that
\begin{align}
\T\vdash\mathrm{D}_{i,\ta_1,\ldots,\ta_n}\,\bigl(\mathrm{A}_\ta\,(B\oslash A)\,\overline{0}\bigr)\,X_1\,\ldots\,X_n\,\overline{d}&=\lambda z^{\ta_j}.\,C \notag\\
\T\vdash\mathrm{D}_{i,\ta_1,\ldots,\ta_n}\,\bigl(\mathrm{A}_\ta\,(B\oslash A)\,\overline{0}\bigr)\,X_1\,\ldots\,X_n\,\overline{d}&=B \notag
\end{align}
\end{proof}

Now define
\begin{align}
\mathrm{E}_\ta^{\N\to\ta}=\lambda x^\N.\,\mathrm{D}_{1,\ta_1,\ta_2,\ldots,\ta_n}\,\bigl(\mathrm{A}_\ta\,x\,\overline{0}\bigr)\,0_{\N\to\ta_1}\,0_{\N\to\ta_2}\,\ldots\,0_{\N\to\ta_n}\,\overline{0} \notag
\end{align}
and note that by Theorem~\ref{t:a}
\begin{align}
\T\vdash\mathrm{E}_\ta\,(A\oslash A)=A \notag
\end{align}
for all pure closed $\beta\eta$-normal forms $A^\ta$.  The term $\mathrm{E}_\ta$ is said to be an \emph{enumerator} for the pure closed $\beta\eta$-normal forms of type $\ta$.

\section{Type Reducibility} \label{s:statman}

The following theorem asserts that each type $\ta$ is $\beta\eta$-reducible to the type $(\N\to\N\to\N)\to\N\to\N$.
\begin{statman}
For each type $\ta$ there exists a pure closed term $\mathrm{M}_\ta$ of type $\ta\to(\N\to\N\to\N)\to\N\to\N$ such that for all pure closed terms $A^\ta$ and $B^\ta$
\begin{align}
\T\vdash\mathrm{M}_\ta\,A=\mathrm{M}_\ta\,B \notag
\end{align}
if and only if
\begin{align}
\T\vdash A=B \notag
\end{align}
\end{statman}
\begin{proof}
See references~\cite{rS80,BDS11}.
\end{proof}

In fact, in the context of the theory $\T$ we can prove a somewhat stronger theorem.
\begin{theorem} \label{t:reduce}
For each type $\ta$ there exists a closed term $\mathrm{N}_\ta^{\ta\to\N}$
in the language of $\T$ such that for all pure closed terms $A^\ta$ and $B^\ta$
\begin{align}
\T\vdash\mathrm{N}_\ta\,A=\mathrm{N}_\ta\,B \notag
\end{align}
if and only if
\begin{align}
\T\vdash A=B \notag
\end{align}
\end{theorem}
\begin{proof}
The type $(\N\to\N\to\N)\to\N\to\N$ is the type of binary trees~\cite{BDS11}.  That is, it can be shown by Lemma~\ref{l:nf} that every pure closed $\beta\eta$-normal form of type $(\N\to\N\to\N)\to\N\to\N$ is of the form $\lambda x^{\N\to\N\to\N}y^\N.\,A^\N$, where the free variables of $A^\N$ are members of the the set $\{x^{\N\to\N\to\N},y^\N\}$.  But again by Lemma~\ref{l:nf}, it must be the case that either $A^\N=y^\N$ or $A^\N$ is of the form $x^{\N\to\N\to\N}\,C^\N\,D^\N$, where the free variables of $C^\N$ and $D^\N$ are members of the the set $\{x^{\N\to\N\to\N},y^\N\}$.  The same argument applies to the subterms $C$ and $D$ themselves.  Hence, $A$ is a binary tree with leaves $y$ and branching nodes $x$.  Furthermore, each tree $A$ can be assigned a numeral $\lVert A\rVert$ by letting $\lVert y\rVert=\overline{0}$, and by letting $\lVert x\,C\,D\rVert=\mathrm{S}_+\,\bigl\langle\lVert C\rVert,\lVert D\rVert\bigr\rangle$.  Note that no two distinct trees are assigned numerals for the same non-negative integer.

Now, for each type $\ta$ define
\begin{align}
\mathrm{N}_\ta^{\ta\to\N}=\lambda x^\ta.\,\mathrm{M}_\ta\,x\,\bigl(\lambda c^\N d^\N.\,\mathrm{S}_+\,\langle c,d\rangle\bigr)\,\overline{0} \notag
\end{align}
By Statman's Type-Reducibility Theorem, for any two distinct pure closed $\beta\eta$-normal forms $A^\ta$ and $B^\ta$, $\mathrm{M}_\ta\,A$ and $\mathrm{M}_\ta\,B$ must have distinct pure closed $\beta\eta$-normal forms of type $(\N\to\N\to\N)\to\N\to\N$.  Then,
\begin{align}
\T\vdash\mathrm{N}_\ta\,A&=\mathrm{M}_\ta\,A\,\bigl(\lambda c^\N d^\N.\,\mathrm{S}_+\,\langle c,d\rangle\bigr)\,\overline{0} \notag\\
\T\vdash\mathrm{N}_\ta\,B&=\mathrm{M}_\ta\,B\,\bigl(\lambda c^\N d^\N.\,\mathrm{S}_+\,\langle c,d\rangle\bigr)\,\overline{0} \notag
\end{align}
Therefore, $\mathrm{N}_\ta\,A$ and $\mathrm{N}_\ta\,B$ have distinct $\beta\eta\T$-normal forms because the $\beta\eta\T$-normal form of $\mathrm{N}_\ta\,A$ is the numeral assigned to the tree in $\mathrm{M}_\ta\,A$, and the $\beta\eta\T$-normal form of $\mathrm{N}_\ta\,B$ is the numeral assigned to the tree in $\mathrm{M}_\ta\,B$.
\end{proof}

\section{Functions of Pure Closed $\beta\eta$-Normal Forms} \label{s:func}

We have described two effective procedures for encoding pure closed $\beta\eta$-normal forms as non-negative integers.  First, a pure closed $\beta\eta$-normal form $A^\ta$ can be encoded as the non-negative integer $n$ such that $\T\vdash\overline{n}=A\oslash A$.  Note that no two distinct pure closed $\beta\eta$-normal forms are encoded as the same non-negative integer, since
\begin{align}
\T\vdash\mathrm{E}_\ta\,(A\oslash A)=A \notag
\end{align}
Alternatively, $A^\ta$ can be encoded as the non-negative integer $n$ such that $\T\vdash\overline{n}=\mathrm{N}_\ta\,A$.  It then follows from Theorem~\ref{t:reduce} that no two distinct pure closed $\beta\eta$-normal forms are encoded as the same non-negative integer.

For any function $f$ from pure closed $\beta\eta$-normal forms of type $\s$ to pure closed $\beta\eta$-normal forms of type $\ta$, define $f_\mathcal{U}$ so that if $f\,A^\s=B^\ta$ then $f_\mathcal{U}\,a=b$ where
\begin{align}
\T\vdash\overline{a}&=A\oslash A \notag\\
\T\vdash\overline{b}&=\mathrm{N}_\ta\,B \notag
\end{align}
We say that $f_\mathcal{U}$ is a $\mathcal{U}$-encoding of $f$.  Similarly, define $f_\mathcal{V}$ so that if $f\,A^\s=B^\ta$ then $f_\mathcal{V}\,a=b$ where
\begin{align}
\T\vdash\overline{a}&=\mathrm{N}_\s\,A \notag\\
\T\vdash\overline{b}&=B\oslash B \notag
\end{align}
We say that $f_\mathcal{V}$ is a $\mathcal{V}$-encoding of $f$.

\begin{theorem} \label{t:u}
Let $f$ be any function from pure closed $\beta\eta$-normal forms of type $\s$ to pure closed $\beta\eta$-normal forms of type $\ta$.  If $f$ can be defined in G\"odel's theory $\T$, then $f$ has a $\mathcal{U}$-encoding that is a $\lep$-recursive function of non-negative integers.
\end{theorem}
\begin{proof}
Consider any function $f$ from pure closed $\beta\eta$-normal forms of type $\s$ to pure closed $\beta\eta$-normal forms of type $\ta$, and suppose that $f$ is defined by a closed term $F^{\s\to\ta}$ in the language of $\T$.  Then $f_\mathcal{U}$ can be defined by the closed term
\begin{align}
\lambda x^\N.\,\mathrm{N}_\ta\,\bigl(F\,(\mathrm{E}_\s\,x)\bigr) \notag
\end{align}
of type $\N\to\N$.  But the closed terms of type $\N\to\N$ in the language of $\T$ define $\lep$-recursive functions of non-negative integers~\cite{gK59,wT65}.  Therefore, $f_\mathcal{U}$ is a $\lep$-recursive function of non-negative integers.
\end{proof}

\begin{theorem} \label{t:v}
Let $f$ be any function from pure closed $\beta\eta$-normal forms of type $\s$ to pure closed $\beta\eta$-normal forms of type $\ta$.  If $f_\mathcal{V}$ is  a $\lep$-recursive function of non-negative integers, then $f$ can be defined in G\"odel's theory $\T$.
\end{theorem}
\begin{proof}
Consider any function $f$ from pure closed $\beta\eta$-normal forms of type $\s$ to pure closed $\beta\eta$-normal forms of type $\ta$, and suppose that $f_\mathcal{V}$ is a $\lep$-recursive function of non-negative integers.  Then~\cite{gK59b}, $f_\mathcal{V}$ can be defined by a closed term $G^{\N\to\N}$ in the language of $\T$.  It immediately follows that $f$ is defined by the term
\begin{align}
\lambda x^\s.\,\mathrm{E}_\ta\,\bigl(G\,(\mathrm{N}_\s\,x)\bigr) \notag
\end{align}
of type $\s\to\ta$.
\end{proof}

Analogs of Theorem~\ref{t:u} and Theorem~\ref{t:v} also hold for extensions of $\T$, such as Spector's theory for bar recursion~\cite{cS62}.  For example, let $f$ be a function from pure closed $\beta\eta$-normal forms of type $\s$ to pure closed $\beta\eta$-normal forms of type $\ta$.  If $f$ can be defined in an extension of $\T$, then $f_\mathcal{U}$ can be defined in that extension of $\T$.  Similarly, if $f_\mathcal{V}$ can be defined in an extension of $\T$, then $f$ can be defined in that extension of $\T$.

\section{Acknowledgments}

A draft of this paper appeared as a chapter in our Ph.D. thesis~\cite{mS10}.  We wish to thank our thesis advisor, Richard Statman, for suggesting the idea of the paper and for his subsequent comments and advice.  We also benefited from comments made by the members of our thesis committee, most notably Jeremy Avigad.

\bibliographystyle{amsplain}
\bibliography{FunctionalsInGodelsT}

\end{document}